\documentclass[12pt]{amsart}
\usepackage{amsmath}
\usepackage{amssymb}
\usepackage{amsfonts} 
\usepackage{mathrsfs}
\usepackage{mathtools}

   \newcommand{\Ndb}{\mbox{$\mathbb{N}$}}
   
   \newcommand{\Pdb}{\mbox{$\mathbb{P}$}}

   \newcommand{\Tdb}{\mbox{$\mathbb{T}$}}

   \newcommand{\C}{\mbox{${\mathcal C}$}}

\newcommand{\norm}[1]{\Vert#1\Vert}
\newcommand{\bignorm}[1]{\bigl\Vert#1\bigr\Vert}
\newcommand{\Bignorm}[1]{\Bigl\Vert#1\Bigr\Vert}
\newcommand{\biggnorm}[1]{\biggl\Vert#1\biggl\Vert}

\newtheorem{theorem}{Theorem}[section]
\newtheorem{lemma}[theorem]{Lemma}
\newtheorem{corollary}[theorem]{Corollary}
\newtheorem{proposition}[theorem]{Proposition}

\theoremstyle{remark}
\newtheorem{remark}[theorem]{\bf Remark}
\theoremstyle{definition}

\numberwithin{equation}{section}

\author[C. Le Merdy]{Christian Le Merdy}
\email{clemerdy@univ-fcomte.fr}
\address{Laboratoire de Math\'ematiques de Besan\c con, UMR 6623,
CNRS, Universit\'e Marie et Louis Pasteur,
25030 Besan\c{c}on Cedex, France}
\author[S. Zadeh]{Safoura Zadeh}
\email{jsafoora@gmail.com}
\address{Institute of Mathematics, University of Bristol, U.K.}

\begin{document}

\title[Random ergodic averages]{Pointwise Convergence for Random Ergodic 
Averages in Non-commutative \(L^p\)-spaces}

\begin{abstract} 
 Let \(M\) be a semifinite von Neumann algebra and \(T\) a positive contraction on both \(L^1(M)\) and \(L^\infty(M)\). We consider ergodic averages along a random sparse subsequence determined by independent Bernoulli variables \((X_n)_{n\geq 1}\) with \(\mathbb{P}(X_n = 1) = n^{-\alpha}\), and set \(W_N = \sum_{n=1}^N \mathbb{E}[X_n]\). We prove that, almost surely, the averages  \(\frac{1}{W_N} \sum_{n=1}^N X_n\, T^n(x)\) converge bilaterally almost uniformly to the ergodic projection for all \(1 < p < \infty\). This extends a theorem of Bourgain to the non-commutative setting.
\end{abstract}

\maketitle



\bigskip
\section{Introduction}\label{1}

In the 1930s, von Neumann proved a monumental mean ergodic theorem, establishing convergence in \(L^2\), while Birkhoff obtained a pointwise ergodic theorem, which guarantees almost everywhere convergence and is often referred to as the individual ergodic theorem. Both results admit a common formulation: for a measure-preserving transformation \(T\), the ergodic averages
\[
A_N(f) := \frac{1}{N}\sum_{n=1}^{N} T^n(f)
\]
converge almost everywhere for all \(f \in L^1\), and converge in \(L^p\) for \(1 < p < \infty\); in general, \(L^1\)-norm convergence fails.
\smallskip

The almost everywhere convergence result was later established within the Dunford--Schwartz ergodic theorem for a broad class of positive operators \(T\), namely positive contractions on both \(L^1\) and \(L^\infty\) (see, for example, \cite[Section 3, Theorem 6]{DS} and \cite[Section 11.1]{EFHN}). A celebrated theorem of Akcoglu and Sucheston \cite{AS} extends this result to all positive contractions on \(L^p\), \(1 < p < \infty\), via a dilation argument. For further details, see Chapters~1 and~5 of \cite{Kren}, along with the notes at the end of those chapters. 

\smallskip
A major direction of research since Birkhoff has been the study of ergodic averages along sparse arithmetic sets. A fundamental example arises when considering polynomial iterates such as  \[ \frac{1}{N}\sum_{n=1}^{N} T^{n^2}(f).\] The increasing gaps and arithmetic rigidity of such sequences prevent direct application of classical ergodic methods. In a seminal work, Bourgain \cite{B2} (see also \cite{B,B7,B5}) established almost everywhere convergence for polynomial sequences with integer coefficients using harmonic-analytic techniques based on discrete Fourier multipliers given by oscillatory exponential sums. These ideas revealed that cancellation in arithmetic phases compensates for sparsity and allows maximal estimates to control pointwise behaviour. A result of comparable nature is the theorem of Bourgain \cite{B5} and Wierdl \cite{W}, which establishes the pointwise ergodic theorem along prime numbers for functions in \( L^p \), \( p > 1 \).
\smallskip

Given these developments, it is natural to ask what happens when the averaging set is no longer deterministic but random. A standard model is provided by sparse random sequences, for instance independent Bernoulli variables \((X_n)\) with \(\mathbb{P}(X_n = 1) = n^{-1/2}\), and the associated random increasing sequence of hitting times
\[
n_k := \min\{N : X_1 + \cdots + X_N = k\}.
\]
In this case, the counting function \(\sum_{n=1}^N X_n\) grows on the order of \(\sqrt{N}\). One may then ask whether the corresponding ergodic averages along this random subsequence,
\[
\frac{1}{m} \sum_{k=1}^m T^{n_k}(f),
\]
still converge almost surely. Bourgain \cite[Appendix~1]{B} showed that, under suitable conditions, almost sure convergence persists even along sufficiently sparse random subsequences.
\medskip

The connection between ergodic theory and von Neumann algebras goes back to the early development of operator algebras. In the non-commutative framework, functions are replaced by operators affiliated with a von Neumann algebra, and classical \(L^p\)-spaces by their non-commutative counterparts. The first ergodic theorem here is due to Lance, who proved an ergodic theorem for trace-preserving \(*\)-automorphisms on a von Neumann algebra. This was sharpened by Yeadon in 1977, who established a non-commutative version of Birkhoff’s pointwise ergodic theorem. Building on these developments, Junge and Xu \cite{JX} obtained a non-commutative extension of the Dunford--Schwartz ergodic theorem. As pointwise convergence no longer makes sense without an underlying space, it is replaced by bilateral almost uniform (b.a.u.) convergence \cite{Y}, in the spirit of Egorov's theorem. Roughly speaking, this means convergence is uniform outside a projection of arbitrarily small trace.
\smallskip

As in the classical theory, maximal inequalities remain central to the study of ergodic behaviour. In the non-commutative setting, however, the absence of a total order rules out any direct notion of a pointwise supremum. To circumvent this, Junge \cite{Ju}, extending Pisier’s construction \cite{P2}, introduced $\ell^\infty$-valued non-commutative $L^p$-spaces, which provide a workable substitute for pointwise suprema and a natural setting for maximal inequalities. The connection with b.a.u. convergence is crucial: much as maximal inequalities can lead to almost everywhere convergence in the commutative case, bounds in these $\ell^\infty$-valued spaces can yield b.a.u. convergence in the non-commutative setting. This principle, introduced in \cite{JX}, is reflected in the closed subspace \(L^p(M;c_0)\) of \(L^p(M;\ell^\infty)\), consisting of sequences that vanish at infinity.
\smallskip

Not every result or method carries over to the non-commutative setting. As noted above, in the commutative case Akcoglu’s dilation theorem allows one to extend the Dunford--Schwartz theorem to general positive contractions. In the non-commutative framework, however, such dilation results fail in general; a counterexample is given in \cite{JLM}. Notably, this example still satisfies a maximal inequality \cite[7.2]{HRW}. 
\smallskip

A dilation is nevertheless available for a restricted class of positive contractions in the noncommutative setting \cite{HRW}, namely those that are separating (i.e. disjointness-preserving), also called Lamperti-type. In the commutative setting these coincide with weighted composition operators, and in the non-commutative case they admit an analogous rigid structural description \cite{HRW,LMZ}. For this class, the dilation method continues to yield both maximal inequalities and norm convergence of the ergodic averages.
\smallskip

With non-commutative analogues of the Birkhoff and Dunford--Schwartz ergodic theorems in place, it is natural to look for a counterpart of Bourgain’s theorem on ergodic averages along polynomial times. In recent work \cite{CHW}, a non-commutative sampling principle is introduced, allowing discrete Fourier multiplier estimates to be transferred to the operator-valued setting. This leads to a maximal inequality for ergodic averages along certain polynomial (arithmetic) sequences in non-commutative \(L^p\)-spaces for \(p > \frac{\sqrt{5}+1}{2}\), extending the harmonic-analytic component of Bourgain’s argument \cite{B7} to the non-commutative framework. However, unlike in the commutative case, it is unknown whether these maximal inequalities imply b.a.u. convergence. Consequently, obtaining a full noncommutative pointwise ergodic theorem by this method remains open. The results of the present paper can be viewed as a step in this direction, establishing b.a.u. convergence for a class of random sparse sequences.
\smallskip

Let \(M\) be a semifinite von Neumann algebra. In this paper, we consider the same random subsequences described above, namely those given by the hitting times \(n_k := \min\{N : X_1 + \cdots + X_N = k\}\). For a positive contraction on both \(M(=L^\infty(M))\) and \(L^1(M)\), we study the corresponding ergodic averages
\begin{equation}\label{RA}
\frac{1}{m} \sum_{k=1}^m T^{n_k}(x),
\end{equation}
and show that for every \(1<p<\infty\) and every \(x\in L^p(M)\) they converge bilaterally almost uniformly to the ergodic projection.  In contrast to the deterministic setting of polynomial times, the argument avoids Fourier-analytic and maximal inequality techniques, relying instead on concentration inequalities. 
\smallskip

The paper is organised as follows. In Section~\ref{2}, we introduce the non-commutative \(L^p\)-framework and collect the main tools used throughout. Section~\ref{3} describes the random model based on the hitting times of the Bernoulli sequence introduced above and establishes its basic probabilistic properties, together with a convenient reformulation of the associated ergodic averages. In Section~\ref{4}, we prove almost sure convergence in \(L^p\)-norm for the random ergodic averages \eqref{RA}. In fact, this convergence holds in the more general setting of an arbitrary reflexive Banach space with non-trivial type for any contraction, not only in \(L^p\)-spaces. Finally, in Section~\ref{5}, we establish b.a.u. convergence for \eqref{RA}. Our results remain valid for non-tracial von Neumann algebras.

\section{Preliminaries on tracial non-commutative $L^p$-spaces}\label{2}

We begin with a brief overview of non-commutative $L^p$-spaces associated with a trace; for further details, see \cite{Hiai,PX}. Let $M$ be a semifinite von Neumann algebra equipped with a normal semifinite faithful trace
$\tau \colon M_+ \to [0,\infty]$. The pair $(M,\tau)$ is called a tracial von Neumann algebra, and we assume that $M \subset B(H)$, acts on a Hilbert space $H$.

Denote by $L^0(M)$ the $*$-algebra of all closed densely defined (possibly unbounded) operators on $H$ that are $\tau$-measurable, and extend $\tau$ to the positive cone $L^0(M)_+$. For $1 \le p < \infty$, the associated non-commutative $L^p$-space is given by
$$ L^p(M) := \{x \in L^0(M) \,:\, \tau(|x|^p) < \infty\}.$$
Equipped with the norm 
$\|x\|_p := \tau(|x|^p)^{1/p}$, $L^p(M)$ is a Banach space. 
For convenience, we set $L^\infty(M) := M$ and let $\|\cdot\|_\infty$ 
denote the operator norm.

For any $x\in L^p(M)$, we have $x^*\in L^p(M)$, with $\norm{x^*}_p=\norm{x}_p$.
We write $L^p(M)_+$ for the set of positive elements in $L^p(M)$. An operator $T \colon L^p(M) \to L^p(M)$ is called positive if it maps $L^p(M)_+$ into itself.

Consider $1 \le p, q, r \le \infty$ with $r^{-1} = p^{-1} + q^{-1}$. For $x \in L^p(M)$ and $y \in L^q(M)$, the product $xy$ lies in $L^r(M)$ and satisfies $\|xy\|_r \le \|x\|_p \|y\|_q$ (Hölder's inequality). The trace $\tau$ extends to a contractive functional on $L^1(M)$, and in the case $q = p' := \frac{p}{p-1}$, $r = 1$, we have $|\tau(xy)| \le \|x\|_p \|y\|_{p'}$. 
If $p<\infty$, the duality pairing
$$
\langle y, x \rangle = \tau(xy), \qquad x \in L^p(M),\ y \in L^{p'}(M),
$$
yields an isometric identification $L^p(M)^* \simeq L^{p'}(M)$. Moreover, $L^2(M)$ is a Hilbert space with inner product 
$(x \vert y)_{L^2} = \tau(y^* x)$ for $x, y \in L^2(M)$.  Finally, for any $1 \le p \le \infty$, 
\begin{equation}\label{2Inter}
L^p(M) = [L^\infty(M), L^1(M)]_{\frac{1}{p}},
\end{equation}
where the right-hand side is defined via complex interpolation. These properties will be used freely throughout the sequel.

Let $1 \le p < \infty$. Following \cite{Ju,JX}, we define $L^p(M;\ell^\infty)$ as the space of all sequences $(u_n)_{n \ge 1}$ in $L^p(M)$ for which there exist $a, b \in L^{2p}(M)$ and a bounded sequence $(v_n)_{n \ge 1}$ in $M$ such that
$$
u_n = a v_n b, \qquad n \ge 1.
$$
For such a sequence, we set
$$
\bigl\|(u_n)_{n \ge 1}\bigr\|_{L^p(M;\ell^\infty)} =
\inf \Bigl\{ \|a\|_{2p} \, \sup_n \|v_n\|_\infty \, \|b\|_{2p} \Bigr\},
$$
where the infimum runs over all possible factorisations of $(u_n)_{n \ge 1}$ in the above form. Equipped with this norm, $L^p(M;\ell^\infty)$ is a Banach space. It is observed in \cite[Remark 3.7]{Ju} and \cite[Section 2]{JX} that if $(u_n)_{n \ge 1}$ is a sequence in $L^p(M)_+$, then $(u_n)_{n \ge 1} \in L^p(M;\ell^\infty)$ if and only if there exists $c \in L^p(M)_+$ such that $u_n \le c$ for all $n \ge 1$. Moreover, in this case, the $L^p(M;\ell^\infty)$-norm of $(u_n)_{n \ge 1}$ coincides with the infimum of $\|c\|_p$ over all such dominating elements $c$. The following is a version of this result for sequences of self-adjoint elements.

\begin{lemma}\label{2sa}
Let $U=(u_n)_{n\geq 1}$ be a sequence of self-adjoint 
elements of $L^p(M)$.
\begin{itemize}
\item [(i)] If $U \in L^p(M;\ell^\infty)$ with $\norm{U}_{L^p(M;\ell^\infty)} < C$ for some $C>0$, then there exists $c \in L^p(M)_+$ with $\norm{c} < C$ such that 
$$
-c \leq u_n \leq c,\qquad n \geq 1.
$$
\item [(ii)] Suppose that there exist $c,d\in L^p(M)_+$ such that 
$$
-d\leq u_n\leq c,\qquad n\geq 1.
$$
Then $U\in L^p(M;\ell^\infty)$ and $\norm{U}_{L^p(M;\ell^\infty)}\leq \norm{c+2d}_p$.
\end{itemize}
\end{lemma}

\begin{proof}
Let ${\rm Ball}(M)$ denote the closed unit ball of $M$. 
We will use the following {\it factorisation principle}, employed in the proof of the triangle inequality in $L^p(M;\ell^\infty)$: 

If $(u_n)_{n\ge 1}$ and $(u'_n)_{n\ge 1}$ are sequences in $L^p(M;\ell^\infty)$ with $L^{2p}(M) \cdotp M \cdotp L^{2p}(M)$ factorisations
$$
u_n = a v_n b, \qquad u'_n = a' v'_n b',
$$
where $v_n, v'_n \in {\rm Ball}(M)$, then there exists a sequence $(w_n)_{n\ge 1}$ in ${\rm Ball}(M)$ such that
$$
u_n + u'_n = \bigl(aa^* + a'a'^*\bigr)^{\frac12} w_n \bigl(b^*b + b'^*b'\bigr)^{\frac12}, \qquad n \ge 1.
$$

\smallskip
(i) By assumption, there exist $a, b \in L^{2p}(M)$ and a sequence $(v_n)_{n\ge 1}$ in ${\rm Ball}(M)$ such that 
$\|a\|_{2p} = \|b\|_{2p} < C^{\frac12}$ and $u_n = a v_n b$ for all $n \ge 1$.  
Since $u_n^* = b^* v_n^* a^*$ and $u_n$ is self-adjoint, the factorisation principle gives
$$
u_n = \frac12(u_n + u_n^*) = \Bigl(\frac{aa^* + b^*b}{2}\Bigr)^{\frac12} w_n \Bigl(\frac{aa^* + b^*b}{2}\Bigr)^{\frac12}, \qquad n \ge 1,
$$
for some $w_n \in M$ with $\|w_n\|_\infty \le 1$.  
Replacing $w_n$ by ${\rm Re}(w_n)$, we may assume that $w_n$ is self-adjoint; 
in this case the norm condition is equivalent to $-1 \le w_n \le 1$.  
Setting
$$
c = \frac{aa^* + b^*b}{2},
$$
we then have $-c \le u_n \le c$ for all $n \ge 1$, and
$$
\|c\|_p \le \frac12 (\|aa^*\|_p + \|b^*b\|_p) = \frac12 (\|a\|_{2p}^2 + \|b\|_{2p}^2) < C.
$$

\smallskip
(ii) The assumption can be rewritten as $0 \le u_n + d \le c + d$.  
By the discussion at the beginning of \cite[Section 2]{JX}, this implies
$$
u_n + d = (c + d)^{\frac12} v_n (c + d)^{\frac12}, \qquad n \ge 1,
$$
for some sequence $(v_n)_{n\ge 1}$ in ${\rm Ball}(M)$.  Writing
$$
u_n = (c+d)^{\frac12} v_n (c+d)^{\frac12} + d^{\frac12}(-1)d^{\frac12}, \qquad n \ge 1,
$$
and applying the factorisation principle, we obtain a sequence 
$(w_n)_{n\ge 1}$ in ${\rm Ball}(M)$ such that
$$
u_n = (c+2d)^{\frac12} w_n (c+2d)^{\frac12}, \qquad n \ge 1.
$$
The result follows immediately.
\end{proof}

Following \cite[Section 6]{JX}, we define $L^p(M;c_0)$ as the space of all sequences $(u_n)_{n \ge 1}$ in $L^p(M)$ for which there exist $a, b \in L^{2p}(M)$ and a bounded sequence $(v_n)_{n \ge 1}$ in ${\rm Ball}(M)$, the closed unit ball of $M$, such that
$$
\lim_n \|v_n\|_\infty = 0 \qquad \text{and} \qquad u_n = a v_n b, \quad n \ge 1.
$$
It turns out that $L^p(M;c_0)$ is a closed subspace of $L^p(M;\ell^\infty)$.

In the sequel, for $n_0 \ge 1$, we identify $(u_n)_{n \ge n_0}$ with the sequence
$$
(0, \ldots, 0, u_{n_0}, u_{n_0+1}, \ldots),
$$
where all entries with index $\le n_0 - 1$ are set to zero.  
Clearly, the norm of $(u_n)_{n \ge n_0}$ in $L^p(M;\ell^\infty)$ decreases as $n_0$ increases.

The closedness of $L^p(M;c_0)$ leads to the following characterisation.

\begin{lemma}\label{2Car-c0}
Let $(u_n)_{n \ge 1}$ be an element of $L^p(M;\ell^\infty)$. Then $(u_n)_{n \ge 1}$ belongs to $L^p(M;c_0)$ if and only if
$$
\bigl\|(u_n)_{n \ge N}\bigr\|_{L^p(M;\ell^\infty)} \longrightarrow 0 \qquad \text{as } N \to \infty.
$$
\end{lemma}

Combining Lemma \ref{2Car-c0} with Lemma \ref{2sa}, we obtain the following.

\begin{corollary}\label{Criteria-sa}
Let $(u_n)_{n \ge 1}$ be a sequence of self-adjoint elements in $L^p(M)$.  
Then $(u_n)_{n \ge 1}$ belongs to $L^p(M;c_0)$ if and only if, for every $\varepsilon > 0$, there exist $n_0 \ge 1$ and $c, d \in L^p(M)_+$ such that
$$
\max\{\|c\|_p, \|d\|_p\} < \varepsilon
\qquad \text{and} \qquad 
-d \le u_n \le c, \quad n \ge n_0.
$$
\end{corollary}

Following \cite[Definition 6.1]{JX}, a sequence $(u_n)_{n \ge 1}$ in $L^p(M)$ is said to converge bilaterally almost uniformly (b.a.u.) to $u \in L^p(M)$ if, for every $\varepsilon > 0$, one can find a projection $e \in M$ such that
$$
\tau(1-e) < \varepsilon \qquad \text{and} \qquad \| e(u_n - u)e \|_\infty \to 0 \quad \text{as } n \to \infty.
$$
Intuitively, this means that $u_n$ approaches $u$ uniformly on a “large part” 
of the algebra, making it the natural non-commutative analogue of almost 
everywhere convergence.  

The next result, from \cite[Lemma 6.2]{JX}, connects this notion to sequences in $L^p(M;c_0)$.

\begin{lemma}\label{2c0-bau}
Let $(u_n)_{n \ge 1}$ be in $L^p(M;c_0)$. Then $u_n \to 0$ b.a.u.
\end{lemma}

\section{Random ergodic averages}\label{3}

In this section, we introduce a Bernoulli sequence \((X_n)\) and associated hitting times \((n_k)\), which yield a sparse random subsequence of natural numbers. We then show that averaging along \((n_k)\) is equivalent to considering weighted sums
\[
\frac{1}{W_N}\sum_{n \leq N} X_n T^n(x).
\]

Let $(\Omega, \mathbb{P})$ be a probability space and fix $\alpha \in (0,1)$.  
Throughout the paper, $(X_n)_{n \ge 1}$ denotes a sequence of independent Bernoulli random variables 
$X_n \colon \Omega \to \{0,1\}$ with success probabilities
\[
\mathbb{P}(X_n = 1) = n^{-\alpha}, \qquad n \ge 1.
\]
In particular, 
\[
\mathbb{E}[X_n] = n^{-\alpha}.
\] 
We set
\[
Y_n := X_n - \mathbb{E}[X_n],
\]
and throughout the paper, $Y_n$ denotes this centered version of $X_n$.

Furthermore, let
\begin{equation}\label{3WN}
W_N := \sum_{n=1}^N n^{-\alpha} \sim \frac{1}{1-\alpha} \, N^{1-\alpha}.
\end{equation}

It follows from the strong law of large numbers that for almost every $\omega\in\Omega$, \begin{equation}\label{3SLLN} X_1+\cdots+X_N\sim W_N, \end{equation} see e.g. \cite[Appendix A]{FLW}.

Consider the hitting times random variables $n_k\colon\Omega\to\Ndb$, for $k\geq 1$, 
defined almost everywhere by 
$$ n_k=\min\bigl\{N\geq 1\, :\, X_1+\cdots +X_N =k\bigr\}. $$ It readily follows 
from (\ref{3WN}), (\ref{3SLLN}) and the identity $X_1+\cdots +X_{n_k}=k$ that for 
almost every $\omega\in\Omega$, 
$$
n_k(\omega)\sim \bigl((1-\alpha)k\bigr)^{\frac{1}{1-\alpha}}. 
$$

The following lemma is straightforward.

\begin{lemma}\label{3Equiv}
Let $X$ be a Banach space, let $(x_n)_{n \ge 1}$ be a sequence in $X$, and let $x \in X$.  
Fix $\omega \in \Omega$ such that
$$
\sum_{n=1}^\infty X_n(\omega) = \infty.
$$
Then the following statements are equivalent:
\begin{itemize}
\item[(i)] 
$$
\frac{1}{m} \sum_{k=1}^m x_{n_k(\omega)} \longrightarrow x \quad \text{as } m \to \infty.
$$
\item[(ii)] 
$$
\frac{1}{\sum_{n=1}^N X_n(\omega)} \sum_{n=1}^N X_n(\omega) x_n \longrightarrow x \quad \text{as } N \to \infty.
$$
\end{itemize}
\end{lemma}

\begin{remark}
Combining the above lemma with \eqref{3SLLN}, we obtain that for any $(x_n)_{n \ge 1}$ in $X$ and any $x \in X$,
$$
\frac{1}{m} \sum_{k=1}^m x_{n_k} \longrightarrow x \quad \text{almost surely}
$$
if and only if
$$
\frac{1}{W_N} \sum_{n=1}^N X_n x_n \longrightarrow x \quad \text{almost surely}.
$$
\end{remark}

Let $X$ be a Banach space and let $I_X$ denote its identity operator.  
For any linear contraction $T \colon X \to X$, set
$$
A_N(T) = \frac{1}{N} \sum_{n=1}^N T^n, \qquad N \ge 1.
$$
If \(X\) is reflexive, then
\begin{equation}\label{3Decomp}
X = \operatorname{Ker}(I_X - T) \oplus \overline{\operatorname{Ran}}(I_X - T),
\end{equation}
by \cite[Theorem 8.22]{EFHN}.
Let \(Q_T \colon X \to X\) denote the projection onto \(\operatorname{Ker}(I_X - T)\) along \(\overline{\operatorname{Ran}}(I_X - T)\). In view of the decomposition \eqref{3Decomp}, it follows that
\begin{equation}\label{3MET}
\lim_{N \to \infty} A_N(T)(x) = Q_T(x), \qquad x \in X.
\end{equation}
The same conclusion holds more generally for any power-bounded operator (i.e., \(\displaystyle{\sup_{n \ge 0} \|T^n\| < \infty}\)).

We consider the random averages given by
\begin{equation}\label{3Sums}
\frac{1}{W_N} \sum_{n=1}^N X_n T^n(x), \qquad N \ge 1,
\end{equation}
for any \(x \in X\). The study of the convergence of such averages \eqref{3Sums}, as well as of the closely related associated averages
\[
\frac{1}{m} \sum_{k=1}^m T^{n_k}(x),
\]
appears in Bourgain's fundamental paper~\cite{B} (see Appendix~1) in the classical $L^p$-setting, as discussed in the introduction; see also \cite[Chapter~5]{K} and the references therein. In this article, we work in the setting $X=L^p(M)$ of non-commutative $L^p$-spaces.

\section{Almost sure norm convergence}\label{4}

In this section, we study almost sure norm convergence of the random ergodic averages associated with the Bernoulli sampling sequence introduced in Section 3. We show that, almost surely in $\omega$, the sequence
\(\frac{1}{W_N}\sum_{n=1}^N X_n(\omega)\, T^n(x)\)
converges in norm to the ergodic projection $Q_T(x)$. The proof splits $X_n$ into its expectation and fluctuation parts, and controls each separately: the fluctuation term vanishes almost surely (Proposition~\ref{4Automatic}), while the deterministic term reduces to a weighted ergodic average (Lemma~\ref{4Cesaro}).
\medskip

Let $X$ be a Banach space.  
For $1 \le p < \infty$, let $L^p(\Omega; X)$ denote the Bochner space of all measurable functions $\phi \colon \Omega \to X$ (defined up to almost everywhere equality) such that the norm function
$t \mapsto \|\phi(t)\|_X$ belongs to $L^p(\Omega)$.  
Equipped with the norm
$$
\|\phi\|_p = \Biggl( \int_{\Omega} \|\phi(t)\|_X^p \, d\Pdb(t) \Biggr)^{\frac{1}{p}},
$$
$L^p(\Omega; X)$ is a Banach space. We refer to \cite[Chapters I--IV]{DU} for the necessary background. 

We will use the classical notion of type of a Banach space; see, for instance, \cite{Maurey}. 
We say that $X$ has non-trivial type if it has type $p$ for some $p \in (1,2]$. 
It is known that non-commutative $L^p$-spaces have type $p$ when $p \in (1,2]$, and type $2$ 
when $p \in [2,\infty)$. We refer to \cite{Fack} for this result; see also 
\cite[Corollary 5.5]{PX}. 
Consequently, Corollary \ref{4NormCV} below applies to non-commutative 
$L^p$-spaces for all $1 < p < \infty$.

The proof of the following result follows arguments developed in \cite[Section 2]{HJP}.
We use $X_n$ and $Y_n$ defined in Section \ref{3}.

\begin{proposition}\label{4Automatic}
Let $X$ be a Banach space with non-trivial type, and let $(x_n)_{n\geq 1}$ be a bounded sequence in $X$. 
Then, for almost every $\omega \in \Omega$,
$$
\frac{1}{W_N} \sum_{n=1}^N Y_n(\omega)\, x_n \longrightarrow 0 \quad \text{as } N \to \infty.
$$
\end{proposition}

\begin{proof}
Assume that $X$ has type $p>1$. By definition of $X_n$ and $Y_n$, we have $Y_n = -\, n^{-\alpha}$ with probability $1-n^{-\alpha}$, and $Y_n = 1-n^{-\alpha}$ with probability $n^{-\alpha}$. Hence,
$$
\|Y_n\|_p^p = (1-n^{-\alpha}) n^{-p\alpha} + n^{-\alpha} (1-n^{-\alpha})^p,
$$
and 
$\|Y_n\|_p^p \sim n^{-\alpha}.$
Applying (\ref{3WN}), we obtain
$$
W_n^{-p} \|Y_n\|_p^p \sim (1-\alpha)^p n^{-p(1-\alpha)-\alpha}.
$$
Since $p>1$, we have $p(1-\alpha) + \alpha > 1$, hence
$$
\sum_{n=1}^\infty W_n^{-p} \|Y_n\|_p^p < \infty.
$$
The sequence $(x_n)_{n\geq 1}$ is bounded and $\|Y_n \otimes x_n\|_p=
\norm{Y_n}_p\norm{x_n}$, so that
$$
\sum_{n=1}^\infty W_n^{-p} \|Y_n \otimes x_n\|_p^p < \infty.
$$
By \cite[(2.1.2)]{HJP} (valid for Banach spaces of type $p$) and Cauchy's criterion, the series
$$
\sum_n W_n^{-1} Y_n \otimes x_n
$$
converges in the Bochner space $L^p(\Omega; X)$. Markov's inequality and \cite[Lemma 1.2]{HJP} then imply that for almost every $\omega \in \Omega$, the series
$$
\sum_n W_n^{-1} Y_n(\omega) x_n
$$
converges in $X$. The result follows from Kronecker's lemma.
\end{proof}

The following result is well-known; we include a proof for completeness, as the proof of Lemma \ref{5Cesaro} follows by a modification of this argument.

\begin{lemma}\label{4Cesaro}
Let $X$ be a Banach space, let $z\in X$, and let $(z_n)_{n\geq 1}$ be a sequence in $X$ such that
\[
\frac{1}{N}\sum_{n=1}^N z_n \longrightarrow z \quad \text{as } N\to\infty.
\]
Then
\[
\frac{1}{W_N}\sum_{n=1}^N n^{-\alpha} z_n \longrightarrow z \quad \text{as } N\to\infty.
\]
\end{lemma}

\begin{proof}
Set $S_m := \sum_{n=1}^m z_n$. By Abel summation, for $N \ge 2$,
\[
\sum_{n=1}^N n^{-\alpha} z_n
= N^{-\alpha} S_N + \sum_{m=1}^{N-1} \bigl(m^{-\alpha} - (m+1)^{-\alpha}\bigr) S_m.
\]
Let
\[
\lambda_m := m \bigl(m^{-\alpha} - (m+1)^{-\alpha}\bigr).
\]
Then
\begin{equation}\label{4Expan}
\frac{1}{W_N} \sum_{n=1}^N n^{-\alpha} z_n
=
\frac{N^{-\alpha}}{W_N} S_N
+
\left(\frac{\sum_{m=1}^{N-1} \lambda_m}{W_N}\right)
\left(\sum_{m=1}^{N-1} \lambda_m\right)^{-1}
\sum_{m=1}^{N-1} \lambda_m \frac{S_m}{m}.
\end{equation}
Since $\lambda_m \sim \alpha m^{-\alpha}$, we have
\begin{equation}\label{4Alpha}
\sum_{m=1}^{N-1} \lambda_m \sim \alpha W_N.
\end{equation}
As $S_m/m \to z$, Cesàro's theorem yields
\[
\left(\sum_{m=1}^{N-1} \lambda_m\right)^{-1}
\sum_{m=1}^{N-1} \lambda_m \frac{S_m}{m} \longrightarrow z,
\]
so the second term in the right-hand side 
of \eqref{4Expan} converges to $\alpha z$. 
By \eqref{3WN}, the first term converges to $(1-\alpha)z$. 
Hence the sum converges to $z$.
\end{proof}

In the following statement, we use the projection $Q_T$ as defined in (\ref{3Decomp}) and (\ref{3MET}).

\begin{corollary}\label{4NormCV}
Let $X$ be a reflexive Banach space with non-trivial type, let
$T\colon X \to X$ be a contraction, and let $x \in X$. Then for almost every $\omega \in \Omega$,
\[
\frac{1}{W_N}\sum_{n=1}^N X_n(\omega)\, T^n(x) \longrightarrow Q_T(x)
\quad \text{as } N\to\infty,
\]
and
\[
\frac{1}{m}\sum_{k=1}^m T^{n_k(\omega)}(x) \longrightarrow Q_T(x)
\quad \text{as } m\to\infty.
\]
\end{corollary}

\begin{proof}
Let $x \in X$. By \eqref{3MET} and Lemma \ref{4Cesaro},
\begin{equation}\label{4CV}
\frac{1}{W_N}\sum_{n=1}^N n^{-\alpha} T^n(x) \longrightarrow Q_T(x).
\end{equation}
Combining \eqref{4CV} with Proposition \ref{4Automatic} yields the first convergence.  The second follows from Lemma \ref{3Equiv}.
\end{proof}

\section{Pointwise convergence}\label{5}

Throughout this section, we let $(M,\tau)$ be a tracial von Neumann algebra.

\subsection{Positive Dunford-Schwartz operators.} We say that an operator $T\colon M\to M$ is a  Dunford-Schwartz operator if $$
\norm{T(x)}_\infty\leq \norm{x}_\infty
\qquad \hbox{and}\qquad
\norm{T(x)}_1\leq \norm{x}_1,
$$
for all $x\in M$. In this case, by interpolation,
the restriction of $T$ to $M\cap L^p(M)$ extends uniquely to a contraction
$$
T_p\colon L^p(M)\longrightarrow L^p(M),
$$
for all $1\leq p\leq \infty$. If, in addition, $T$ is positive,
then $T_p$ is positive for all $1\leq p\leq \infty$.

Let $T$ be as above and fix $1<p<\infty$.
Since $L^p(M)$ is reflexive, we may consider the projection
$Q_{T_p}\colon L^p(M)\to L^p(M)$
provided by (\ref{3Decomp}) and (\ref{3MET}).
Then, according to \cite[Theorem 6.3]{JX}, we have
\begin{equation}\label{5c0}
\Bigl(\frac{1}{N}\sum_{n=1}^N T_p^n(x)
- Q_{T_p}(x)\Bigr)_{N\geq1}\,\in L^p(M;c_0),
\end{equation}
for all $x\in L^p(M)$. Adapting the argument in the proof of 
Lemma \ref{4Cesaro}, we obtain the following.

\begin{lemma}\label{5Cesaro}
For every $x\in L^p(M)$, we have
$$
\Bigl(\frac{1}{W_N}\sum_{n=1}^N n^{-\alpha} T_p^n(x)
- Q_{T_p}(x)\Bigr)_{N\geq1}\in L^p(M;c_0).
$$
\end{lemma}

\begin{proof} 
We write $T=T_p$ and $Q=Q_{T_p}$ for simplicity.
For $n\geq 1$, set $z_n = T^n(x) - Q(x)$, and for $m\geq 1$, set
$S_m = \sum_{n=1}^m z_n$. Then, by (\ref{5c0}),
\begin{equation}\label{5c0bis}
\Bigl(\frac{S_N}{N}\Bigr)_{N\geq 1}\,\in L^p(M;c_0).
\end{equation}
Furthermore, the computation in the proof of
Lemma \ref{4Cesaro} yields (\ref{4Expan}) for all 
$N\geq 2$.

We treat the two terms on the right-hand side of (\ref{4Expan}) separately.
First, using (\ref{3WN}) and (\ref{5c0bis}), we obtain
$$
\Bigl(\frac{N^{-\alpha}}{W_N}\, S_N \Bigr)_{N\geq 1}\,\in L^p(M;c_0).
$$
Second, we claim that
$$
\biggl(\Bigl(\sum_{m=1}^{N-1}\lambda_m\Bigr)^{-1}\cdotp \sum_{m=1}^{N-1}
\lambda_m\Bigl(\frac{S_m}{m}\Bigr)\biggr)_{N\geq 1}\, \in L^p(M;c_0).
$$
To verify this, write
$$
\frac{S_m}{m} = ay_mb,\qquad m\geq 1,
$$
with $a,b\in L^{2p}(M)$ and $y_m\in M$ such that $y_m\to 0$.
Then, for any $N\geq 1$, we have the $L^{2p}(M)\cdotp M\cdotp L^{2p}(M)$ decomposition
$$
\biggl(\sum_{m=1}^{N-1}\lambda_m\biggr)^{-1}\cdotp \sum_{m=1}^{N-1}
\lambda_m\Bigl(\frac{S_m}{m}\Bigr)
= a\biggl(\biggl(\sum_{m=1}^{N-1}\lambda_m\biggr)^{-1}
\cdotp \sum_{m=1}^{N-1} \lambda_m y_m\biggr)b,
\qquad N\geq 1.
$$
Since $\sum_{m=1}^{N-1}\lambda_m\to\infty$, Cesàro's lemma ensures that
$$
\biggl(\sum_{m=1}^{N-1}\lambda_m\biggr)^{-1}\cdotp \sum_{m=1}^{N-1}
\lambda_m y_m
\,\longrightarrow\, 0,
$$
as $N\to \infty$. This proves the claim. Then, using (\ref{4Alpha}), we deduce that the sequence given by the
second term on the right-hand side of (\ref{4Expan}) belongs to $L^p(M;c_0)$. Therefore,
$$
\biggl(\frac{1}{W_N}\sum_{n=1}^N n^{-\alpha} z_n\biggr)_{N\geq 1}
\,\in L^p(M;c_0),
$$
which yields the result.
\end{proof}

\subsection{Almost sure convergence of random ergodic averages}

If $T\colon M\to M$ is any
positive Dunford-Schwartz operator, then the Junge-Xu individual ergodic theorem \cite[Corollary 6.4]{JX} asserts that for all $1<p<\infty$ and all $x\in L^p(M)$, 
$$
\frac{1}{N}\,\sum_{n=1}^N T_p^n(x)\,\longrightarrow Q_{T_p}(x)\quad\hbox{b.a.u.}
$$

The main objective of this section is to establish a probabilistic version of this result.

\begin{theorem}\label{5Main}
Let $T\colon M\to M$ be a positive Dunford-Schwartz operator, and let $1<p<\infty$.
Then there exists a measurable set $\Omega'\subset \Omega$ with $\Pdb(\Omega')=1$ such that, for every $\omega\in\Omega'$ and every $x\in L^p(M)$, the following hold:
\begin{itemize}
    \item[(1)]
    $$
    \frac{1}{W_N}\sum_{n=1}^N X_n(\omega)\, T_p^n(x)
    \,\longrightarrow Q_{T_p}(x)\quad\hbox{b.a.u.}
    \quad\hbox{as } N\to\infty.
    $$
    \item[(2)]
    $$
    \frac{1}{m}\sum_{k=1}^m T_p^{n_k(\omega)}(x)
    \,\longrightarrow Q_{T_p}(x)\quad\hbox{b.a.u.}
    \quad\hbox{as } m\to\infty.
    $$
\end{itemize}
\end{theorem}

The idea is to decompose the random ergodic averages into a deterministic part and a fluctuation part
\[
X_n = n^{-\alpha} + Y_n,
\quad \text{where } Y_n = X_n - \mathbb{E}[X_n],
\]
and treat these contributions separately. To control the fluctuation part, we establish a strengthening of Corollary~\ref{4NormCV} in the Hilbert space setting, together with a weaker version of Proposition~\ref{4Automatic} that holds without any geometric assumption on $X$. 

\begin{proposition}\label{5Hilbert}
Let $H$ be a Hilbert space, and let $T\colon H\to H$ be a contraction. 
There exists $\varepsilon>0$ such that the following holds.
For almost every $\omega\in \Omega$, there exists a constant
$C_\omega\geq 0$ such that 
\begin{equation}\label{Lac}
\forall\, N\geq 1,\qquad \frac{1}{W_{N}}\,
\Bignorm{\sum_{n=1}^N
Y_n(\omega) T^n}_{H\to H}\leq C_\omega N^{-\varepsilon}.
\end{equation}
\end{proposition}

\begin{proof}
Let $\Tdb$ denote the unit circle, equipped with its normalized Haar measure, and for any integer $n\geq 1$, let $e_n\in L^\infty(\Tdb)$ be defined by $e_n(\xi)=\xi^n$ for all $\xi\in\Tdb$. Fix $N\geq 1$ and define $\Gamma_N\colon \Omega\to L^\infty(\Tdb)$ by 
$$
\Gamma_N(\omega) = \frac{1}{W_N}\sum_{n=1}^N Y_n(\omega)e_n.
$$
Let $\Delta_N$ be a $\frac{1}{2N}$-net of $\Tdb$ with
${\rm Card}(\Delta_N)\leq 4\pi N$. 
Given any $\gamma\in\Tdb$, there exists $\xi\in\Delta_N$ such that 
$\vert \xi-\gamma\vert\leq\frac{1}{2N}$. Then, by Bernstein's 
theorem and the Mean Value theorem,
\begin{align*}
\Bigl\vert\sum_{n=1}^N Y_n(\omega) \gamma^n
- \sum_{n=1}^N Y_n(\omega) \xi^n\Bigr\vert
&\leq N\vert \xi-\gamma\vert
\Bignorm{\sum_{n=1}^N Y_n(\omega) e_n}_\infty\\
&\leq \frac12 
\Bignorm{\sum_{n=1}^N Y_n(\omega) e_n}_\infty,
\end{align*}
for all $\omega\in\Omega$.
This implies
$$
\Bigl\vert\sum_{n=1}^N Y_n(\omega) \gamma^n\Bigr\vert\leq
\frac12 
\Bignorm{\sum_{n=1}^N Y_n(\omega) e_n}_\infty
+ \Bigl\vert\sum_{n=1}^N Y_n(\omega) \xi^n\Bigr\vert,
$$
and hence
$$
\Bignorm{\sum_{n=1}^N Y_n(\omega) e_n}_\infty\leq 
2 \sup\Bigl\{\Bigl\vert\sum_{n=1}^N Y_n(\omega) \xi^n\Bigr\vert\,:\, \xi\in\Delta_N\Bigr\},
$$
for all $\omega\in\Omega$. We deduce from this estimate that 
\begin{align*}
\Pdb\Bigl(\bignorm{\Gamma_N(\,\cdotp)}_\infty\geq W_N^{-\frac14} \Bigr)
& \leq
\Pdb\biggl(\sup_{\xi\in\Delta_N}
\Bigl\vert\sum_{n=1}^N Y_n(\,\cdotp) \xi^n\Bigr\vert
\geq \frac{W_N^\frac34}{2}\biggr)\\
& \leq \sum_{\xi\in\Delta_N} \Pdb\biggl(
\Bigl\vert\sum_{n=1}^N Y_n(\,\cdotp) \xi^n\Bigr\vert\geq 
\frac{W_N^\frac34}{2}\biggr)\\
&\leq (4\pi N) \sup_{\xi\in\Delta_N}\Pdb\biggl(\Bigl\vert\sum_{n=1}^N Y_n(\,\cdotp) \xi^n\Bigr\vert\geq 
\frac{W_N^\frac34}{2}\biggr).
\end{align*}

Fix $\xi\in \Delta_N$. Since the
$Y_n$ are centered and independent, we have
$$
\sigma_N:=
{\rm Var}\Bigl(\sum_{n=1}^N
Y_n\xi^n\Bigr)^\frac12=\Bignorm{\sum_{n=1}^N
Y_n\xi^n}_2 = \Bigl(\sum_{n=1}^N \norm{Y_n}_2^2\Bigr)^\frac12.
$$
A computation in the proof of Proposition \ref{4Automatic}
shows that $\norm{Y_n}_2^2\sim n^{-\alpha}$, and hence 
$\sigma_N\sim W_N^\frac12.$
Set $\lambda_N=(2\sigma_N)^{-1}W_N^{\frac34}$. Then
\begin{equation}\label{5LambdaN}
\lambda_N^2\sim \frac{W_N^\frac12}{4}.
\end{equation}
Moreover, by Chernoff's inequality
(see e.g. \cite[Theorem 1.8]{Tao}), 
$$
\Pdb\biggl(\Bigl\vert\sum_{n=1}^N Y_n(\,\cdotp) \xi^n\Bigr\vert\geq 
\frac{W_N^\frac34}{2}\biggr)
\leq 2\max\Bigl\{e^{-\frac{\lambda_N^2}{4}},
e^{-\frac{W_N^\frac34}{4}}\Bigr\}.
$$
Hence, for $N$ large enough,
$$
\Pdb\biggl(\Bigl\vert\sum_{n=1}^N Y_n(\,\cdotp) \xi^n\Bigr\vert\geq 
\frac{W_N^\frac34}{2}\biggr)
\leq 2\,e^{-\frac{\lambda_N^2}{4}}.
$$
Taking the supremum over $\xi\in\Delta_N$, we obtain
$$
\Pdb\Bigl(\bignorm{\Gamma_N(\,\cdotp)}_\infty\geq W_N^{-\frac14} \Bigr)
\lesssim N \,e^{-\frac{\lambda_N^2}{4}}.
$$
By (\ref{5LambdaN}) and (\ref{3WN}), this implies that
$$
\sum_{N=1}^\infty\Pdb\Bigl( \bignorm{\Gamma_N(\,\cdotp)}_\infty\geq W_N^{-\frac14}  \Bigr)\,<\infty.
$$
Let $\varepsilon =(1-\alpha)/4$. 
By the Borel--Cantelli lemma and (\ref{3WN}), we deduce that for almost every $\omega\in \Omega$, there exists a constant
$C_\omega\geq 0$ such that 
$$
\forall\, N\geq 1,\qquad \frac{1}{W_{N}}\,\Bignorm{
\sum_{n=1}^N
Y_n(\omega) e_n}_{L^\infty(\mathbb T)}\leq C_\omega N^{-\varepsilon}.
$$
Finally, by von Neumann's inequality,
$$
\Bignorm{\sum_{n=1}^N Y_n(\omega) T^n}_{H\to H}
\leq \Bignorm{
\sum_{n=1}^N
Y_n(\omega) e_n}_{L^\infty(\mathbb T)},
\qquad\omega\in\Omega.
$$
This yields the result.
\end{proof}

\begin{remark} 
Assume that $T\colon H\to H$ 
is a power bounded operator on a Hilbert space. By \cite[Corollary 3.9]{Pe}, 
there exists a constant $C\geq 1$ such that 
$$
\Bignorm{\sum_{n=1}^N Y_n(\omega) T^n}_{H\to H}
\leq C\,{\rm Log}(N)\,\Bignorm{
\sum_{n=1}^N
Y_n(\omega) e_n}_{L^\infty(\mathbb T)},
$$
for all $N\geq 2$ and all $\omega\in\Omega$.
Since ${\rm Log}(N)\, N^{-\varepsilon}\lesssim 
N^{-\frac{\varepsilon}{2}}$, the above proof shows that Proposition \ref{5Hilbert} remains valid
for power bounded operators.
\end{remark}

\begin{lemma}\label{5L1}
Let $X$ be a Banach space and let $(x_n)_{n\geq 1}$ be a bounded sequence in $X$. 
For almost every $\omega\in \Omega$, there exists a constant
$C_\omega\geq 0$ such that 
$$
\forall\, N\geq 1,\qquad \frac{1}{W_N}\,\Bignorm{\sum_{n=1}^N Y_n(\omega) x_n}\leq C_\omega.
$$
\end{lemma}

\begin{proof}
Recall that $W_N$ is the expectation 
of $X_1+\cdots+X_N$. Hence, by the 
multiplicative Chernoff bound (see e.g. \cite[Theorem 4.1]{MR}), 
we have
$$
\Pdb\Bigl(\sum_{n=1}^N X_n\geq (1+\delta)
W_N\Bigr)\leq\biggl(\frac{e^\delta}{(1+\delta)^{1+\delta}}\biggr)^{W_N},
$$
for all $\delta>0$ and all $N\geq 1$.
Choose $\delta$ such that 
$$
a := \frac{e^\delta}{(1+\delta)^{1+\delta}} < 1.
$$
By (\ref{3WN}), we have $\sum_{N=1}^\infty a^{W_N}<\infty$. Hence,
$$
\sum_{N=1}^\infty \Pdb \Bigl(\frac{1}{W_N}\,\sum_{n=1}^N X_n\geq (1+\delta)\Bigr)<\infty.
$$
By the Borel--Cantelli lemma, 
this implies that for almost every $\omega\in\Omega$, the sequence
\begin{equation}\label{Bdd}
\Bigl(\frac{1}{W_N}\,\sum_{n=1}^N X_n(\omega)\Bigr)_{N\geq 1}
\quad\hbox{}
\end{equation}
is bounded. Using the boundedness of $(x_n)_{n\geq 1}$, we write
\begin{align*}
\frac{1}{W_N}\,\Bignorm{\sum_{n=1}^N Y_n(\omega) x_n}
& \lesssim \frac{1}{W_N}\,\sum_{n=1}^N \vert Y_n(\omega)\vert\\
& \lesssim \frac{1}{W_N}\,\sum_{n=1}^N \bigl(X_n(\omega) + n^{-\alpha}\bigr) \\
& \lesssim \Biggl(\frac{1}{W_N}\,\sum_{n=1}^N X_n(\omega)\Biggr) + 1.
\end{align*}
The result now follows from (\ref{Bdd}).
\end{proof}

\begin{proof}[Proof of Theorem \ref{5Main}]
Let $T$ be as in Theorem~\ref{5Main}, and fix $1<p<\infty$. 
If $1<p\leq 2$, then by interpolation,
\[
\Big\|\sum_{n=1}^N Y_n(\omega) T_p^n\Big\|_{L^p\to L^p}
\leq 
\Big\|\sum_{n=1}^N Y_n(\omega) T_p^n\Big\|_{L^1\to L^1}^{2(\frac1p -\frac12)}
\Big\|\sum_{n=1}^N Y_n(\omega) T_p^n\Big\|_{L^2\to L^2}^{2(1-\frac1p)},
\]
for all $N\geq 1$ and $\omega\in\Omega$.  
Applying Proposition~\ref{5Hilbert} with $H=L^2(M)$ and Lemma~\ref{5L1} with $X=L^1(M)$, we obtain the existence of $\varepsilon>0$ and a measurable set $\Omega'\subset\Omega$ with $\Pdb(\Omega')=1$ such that, for every $\omega\in\Omega'$, there exists a constant $C_\omega\geq 0$ for which
\begin{equation}\label{5LacLp}
\forall\, N\geq 1,\qquad 
\frac{1}{W_N}\,
\Big\|\sum_{n=1}^N Y_n(\omega) T_p^n\Big\|_{L^p\to L^p}
\leq C_\omega N^{-\varepsilon}.
\end{equation}
An analogous argument applies when $2\leq p<\infty$, using interpolation between $L^2(M)$ and $L^\infty(M)$. 

From now on, we write $T = T_p$ and $Q = Q_{T_p}$ for simplicity. Our main objective is to prove that for every $\omega \in \Omega'$ and every $x \in L^p(M)$,
\begin{equation}\label{5Goal}
\left(\frac{1}{W_N}\sum_{n=1}^N X_n(\omega) T^n(x) - Q(x)\right)_{N\geq 1}
\in L^p(M; c_0).
\end{equation}

Fix $x\in L^p(M)$ and $\omega\in\Omega'$. We consider the sequence
\[
\biggl(\frac{1}{W_{N}}\,
\sum_{n=1}^N
Y_n(\omega) T^n(x)\biggr)_{N\geq 1}
\]
and introduce lacunary subsequences in order to exploit the decay provided by (\ref{5LacLp}). Although this decay is not summable over all $N$, it becomes summable along lacunary subsequences. More precisely, for an arbitrary $r>1$, we define 
\[
u_k = 
\frac{1}{W_{\lfloor r^k \rfloor}}\,
\sum_{n=1}^{\lfloor r^k \rfloor} 
Y_n(\omega) T^n(x),\qquad k\geq 1.
\]
Let $(\epsilon_k)_{k\geq 1}$ denote the standard Schauder basis of $c_0$. Then 
\[
\norm{u_k\otimes \epsilon_k}_{L^p(M;\ell^\infty)}
=\norm{u_k}
\]
for all $k\geq 1$, and hence, by (\ref{5LacLp}),
\[
\sum_{k=1}^\infty \norm{u_k\otimes \epsilon_k}_{L^p(M;\ell^\infty)}<\infty.
\]
Since $L^p(M;c_0)$ is closed, it follows that the sum
\[
v:=\sum_{k=1}^\infty u_k\otimes \epsilon_k
\]
belongs to $L^p(M;c_0)$. Clearly, $v=(u_k)_{k\geq 1}$, and therefore
\begin{equation}\label{5Lacun}
\biggl(\frac{1}{W_{\lfloor r^k \rfloor}}\,
\sum_{n=1}^{\lfloor r^k \rfloor} 
Y_n(\omega) T^n(x)\biggr)_{k\geq 1}
\in L^p(M;c_0).
\end{equation}

We now turn to the proof of (\ref{5Goal}). The idea is to pass from the control obtained along the lacunary subsequence (\ref{5Lacun}) to the full sequence by comparing each index $N$ with the neighbouring lacunary indices.

Since $L^p(M)$ is spanned by its positive elements, we may assume that $x\in L^p(M)_+$. We set $x_n =T^n(x)$ for all $n\geq 1$, so that $x_n\geq 0$.

Let $\eta>0$. By Lemmas \ref{5Cesaro} and \ref{2Car-c0}, there exists $N_0\geq 1$ such that
\[
\Big\|\Bigl(\frac{1}{W_N}\sum_{n=1}^N n^{-\alpha}x_n -Q(x)\Bigr)_{N\geq N_0}\Big\|_{L^p(M;\ell^\infty)}<\eta.
\]
By Lemma \ref{2sa}, (1), this yields the existence of $c\in L^p(M)_+$ with $\norm{c}<\eta$ such that
\[
-c\leq \frac{1}{W_N}\sum_{n=1}^N n^{-\alpha}x_n -Q(x) \leq c,\qquad N\geq N_0.
\]
According to (\ref{3WN}), we may assume that $N_0$ is large enough so that
\begin{equation}\label{Eps1}
\frac{1}{1+\eta}(1-\alpha)\leq \frac{N^{1-\alpha}}{W_N}\leq (1+\eta)(1-\alpha),\qquad N\geq N_0.
\end{equation}
We choose $r>1$ such that
\begin{equation}\label{Eps2}
r^{1-\alpha}\leq 1+\eta,
\end{equation}
and apply (\ref{5Lacun}) with this value of $r$. By Lemma \ref{2Car-c0}, there exists $k_0\geq 1$ such that
\[
\Big\|\Bigl(\frac{1}{W_{\lfloor r^k \rfloor}}
\sum_{n\leq \lfloor r^k \rfloor} Y_n(\omega)x_n\Bigr)_{k\geq k_0}\Big\|_{L^p(M;\ell^\infty)}<\eta.
\]
By Lemma \ref{2sa}, (1), there exists $d\in L^p(M)_+$ with $\norm{d}_p<\eta$ such that
\[
-d\leq \frac{1}{W_{\lfloor r^k \rfloor}}
\sum_{n\leq \lfloor r^k \rfloor} Y_n(\omega)x_n \leq d,\qquad k\geq k_0.
\]
We may assume that $k_0$ is large enough so that $\lfloor r^k \rfloor<\lfloor r^{k+1} \rfloor$ for all $k\geq k_0$, and we set $\widetilde{N}_0=\max\{r^{k_0},N_0\}$.

Let $N\geq \widetilde{N}_0$, and let $k\geq k_0$ be such that
\[
\lfloor r^k \rfloor \leq N < \lfloor r^{k+1} \rfloor.
\]
Since $X_n(\omega)\in\{0,1\}$ and $x_n\geq 0$, we have
\[
\sum_{n=1}^N X_n(\omega)x_n \leq \sum_{n=1}^{\lfloor r^{k+1}\rfloor} X_n(\omega)x_n,
\]
and therefore
\[
\frac{1}{W_N}\sum_{n=1}^N X_n(\omega)x_n
\leq \frac{W_{\lfloor r^{k+1}\rfloor}}{W_N}
\Bigl(\frac{1}{W_{\lfloor r^{k+1}\rfloor}}
\sum_{n=1}^{\lfloor r^{k+1}\rfloor} X_n(\omega)x_n\Bigr).
\]
On the one hand,
\[
\frac{W_{\lfloor r^{k+1}\rfloor}}{W_N}
\leq \frac{W_{\lfloor r^{k+1}\rfloor}}{W_{\lfloor r^k\rfloor}}
\leq (1+\eta)^3,
\]
by (\ref{Eps1}) and (\ref{Eps2}). On the other hand,
\begin{align*}
\frac{1}{W_{\lfloor r^{k+1}\rfloor}}
\sum_{n=1}^{\lfloor r^{k+1}\rfloor} X_n(\omega)x_n
&= \frac{1}{W_{\lfloor r^{k+1}\rfloor}}
\sum_{n=1}^{\lfloor r^{k+1}\rfloor} Y_n(\omega)x_n
+ \frac{1}{W_{\lfloor r^{k+1}\rfloor}}
\sum_{n=1}^{\lfloor r^{k+1}\rfloor} n^{-\alpha}x_n\\
&\leq d+c+Q(x).
\end{align*}
Hence
\[
\frac{1}{W_N}\sum_{n=1}^N X_n(\omega)x_n
\leq (1+\eta)^3(d+c+Q(x)).
\]
A similar argument yields
\[
\frac{1}{W_N}\sum_{n=1}^N X_n(\omega)x_n
\geq -(1+\eta)^{-3}(d+c-Q(x)).
\]
Subtracting $Q(x)$ gives
\[
-\gamma \leq \frac{1}{W_N}\sum_{n=1}^N X_n(\omega)x_n - Q(x) \leq \gamma,
\]
where
\[
\gamma = (1+\eta)^3(d+c) + ((1+\eta)^3-1)Q(x).
\]
By Lemma \ref{2sa}, (ii),
\begin{align*}
\Big\|\Bigl(\frac{1}{W_N}\sum_{n=1}^N X_n(\omega)x_n - Q(x)\Bigr)_{N\geq \widetilde{N}_0}\Big\|_{L^p(M;\ell^\infty)}
&\leq 3\norm{\gamma}_p\\
&\leq 3\Bigl(2\eta(1+\eta)^3 + ((1+\eta)^3-1)\norm{x}_p\Bigr).
\end{align*}
Since this bound tends to $0$ as $\eta\to 0$, we obtain
\[
\Big\|\Bigl(\frac{1}{W_N}\sum_{n=1}^N X_n(\omega)x_n - Q(x)\Bigr)_{N\geq \widetilde{N}}\Big\|_{L^p(M;\ell^\infty)}
\longrightarrow 0 \quad \text{as } \widetilde{N}\to\infty.
\]
By Lemma \ref{2Car-c0}, this proves (\ref{5Goal}).

We may now deduce part (1) of the theorem by 
applying Lemma \ref{2c0-bau}. To prove part (2), we combine 
(\ref{5Goal}) and (\ref{3SLLN}). From these two properties,
we obtain a measurable set
$\Omega''\subset\Omega$ with $\Pdb(\Omega'')=1$, such that
we both have
$$
\biggl(\frac{1}{\sum_{n=1}^N X_n(\omega)}\,
\sum_{n=1}^N
X_n(\omega) T^n(x)\, -Q(x)\biggr)_{N\geq 1}\,\in L^p(M;c_0),
\qquad \omega\in\Omega'',\,
x\in L^p(M),
$$
and
$$
\sum_{n=1}^\infty X_n(\omega)\,=\infty,\qquad \omega\in\Omega''.
$$
Let $x\in L^p(M)$ and let $\omega\in\Omega''$. 
Let $\varepsilon>0$. By Lemma \ref{2c0-bau},
there exists a projection $e\in M$ such that $\tau(1-e)<\varepsilon$
and 
$$
\biggnorm{e\biggl(\frac{1}{\sum_{n=1}^N X_n(\omega)} 
\sum_{n=1}^N X_n(\omega) T^n(x) \, - Q(x)
\biggr)e}_\infty\longrightarrow 0.
$$
This can be rephrased as
$$
\biggnorm{\frac{1}{\sum_{n=1}^N X_n(\omega)} 
\sum_{n=1}^N X_n(\omega) e\bigl(T^n(x)  - Q(x)\bigr)e
}_\infty\longrightarrow 0.
$$
Applying  Lemma \ref{3Equiv} in $X=M$, we deduce that
$$
\biggnorm{e\biggl(\frac{1}{m} 
\sum_{k=1}^m T^{n_k(\omega)}(x) \, - Q(x)\biggr)e
}_\infty = 
\biggnorm{\frac{1}{m} 
\sum_{k=1}^m e\bigl(T^{n_k(\omega)}(x) \, - Q(x)\bigr)e
}_\infty\longrightarrow 0,
$$
This show that $\frac{1}{m} 
\sum_{k=1}^m T^{n_k(\omega)}(x)\to Q(x)$ b.a.u. as expected.
\end{proof}

\begin{remark}
This remark explains how
Theorem \ref{5Main} extends to the non-tracial setting. 
In this remark (and only here), we consider non-tracial
$L^p$-spaces. Let $M$ be a von Neumann algebra and let
$\phi\colon M_+\to[0,\infty]$ be a normal semifinite faithful weight.
For $1\leq p\leq\infty$, let $L^p(M,\phi)$ denote the Haagerup
non-commutative $L^p$-space associated with $(M,\phi)$, see
\cite{HJX,Hiai,Terp0}. The space $L^p(M,\phi;c_0)$ is defined as in the tracial case.

Let $T\colon M\to M$ be a completely positive contraction such that $\phi\circ T\leq\phi$ on $M_+$. For $1\leq p<\infty$, $T$ extends to a positive contraction
$T_p\colon L^p(M,\phi)\to L^p(M,\phi)$ by \cite[Theorem 5.1 $\&$ Remark 5.6]{HJX}.
Assume that $T$ commutes with the modular automorphism group of $\phi$.

Let $1<p<\infty$, and let $Q_{T_p}$ be the projection defined by (\ref{3Decomp}) and (\ref{3MET}). By \cite[Theorem 7.11]{JX}, for all $x\in L^p(M,\phi)$,
\[
\bigl(A_N(T_p)(x) - Q_{T_p}(x)\bigr)_{N\geq 1} \in L^p(M,\phi;c_0).
\]
With this property at hand, (the proof of) Theorem \ref{5Main} carries over to the non-tracial setting. In particular, almost surely, for all $x\in L^p(M,\phi)$,
\begin{equation}\label{5NTc0}
\biggl(\frac{1}{W_N}\sum_{n=1}^N X_n(\omega)\, T_p^n(x) - Q_{T_p}(x)\biggr)_{N\geq 1}
\in L^p(M,\phi;c_0).
\end{equation}

We refer to \cite[Definition 7.9]{JX} for b.a.u. convergence in the non-tracial setting. Combining (\ref{5NTc0}) with
\cite[Lemma 7.10]{JX}, we obtain a measurable set $\Omega'\subset \Omega$ with $\Pdb(\Omega')=1$ such that properties (1) and (2) of Theorem \ref{5Main} hold.

\end{remark}

\subsection{Separating positive contractions}
In this subsection, we study operators acting on a single non-commutative $L^p$-space, in contrast with the setting of Theorem \ref{5Main}.

Let $1<p<\infty$. We say that two elements $x, y \in L^p(M)$ are disjoint if $x^*y = xy^* = 0$. An operator $T \colon L^p(M) \to L^p(M)$ is called separating if it preserves disjointness; that is, whenever $x$ and $y$ are disjoint, so are $T(x)$ and $T(y)$. Such operators are often referred to as Lamperti operators in the literature. We refer to \cite{HRW,LMZ} for background and, in particular, for characterisations of separating operators (see \cite[Theorem 3.1]{LMZ} and \cite[Theorem 3.3]{HRW}).

We define
$$
\C_p = \overline{\rm Conv}^{SOT} \Bigl\{ T \colon L^p(M) \to L^p(M) \,\big|\, 
T \text{ is positive, contractive, and separating} \Bigr\},
$$
where ${\rm Conv}$ denotes the convex hull and $\overline{\rm Conv}^{SOT}$ its closure in the strong operator topology.

\begin{theorem}\label{6SepMain}
Let $T\in\C_p$ and let $Q_T$ be the projection defined by (\ref{3Decomp}) and (\ref{3MET}). Then there exists a measurable set $\Omega'\subset \Omega$ with $\Pdb(\Omega')=1$ such that, for every $\omega\in\Omega'$, the following hold:

\begin{itemize}
\item[(1)] For all $x\in L^p(M)$,
$$
\frac{1}{W_N}\sum_{n=1}^N X_n(\omega)\, T^n(x)\,\longrightarrow Q_{T}(x) 
\quad\text{b.a.u. as } N\to\infty.
$$

\item[(2)] For all $x\in L^p(M)$,
$$
\frac{1}{m}\sum_{k=1}^m T^{n_k(\omega)}(x)\,\longrightarrow Q_{T}(x) 
\quad\text{b.a.u. as } m\to\infty.
$$
\end{itemize}
\end{theorem}
The proof will rely on the following transference lemma. 
In this statement, the isometry $V$ is not assumed to be surjective.

For any Banach space $Y$, let 
${\mathfrak s}_{p,Y}\colon \ell^p_{\mathbb Z}(Y) \to \ell^p_{\mathbb Z}(Y)$ 
denote the backward shift operator, defined by
$$
{\mathfrak s}_{p,Y}\bigl((b_k)_k\bigr) = (b_{k+1})_k, 
\qquad (b_k)_k \in \ell^p_{\mathbb Z}(Y).
$$

\begin{lemma}\label{6Transfer}
Let $V\in B(Y)$ be an isometry (not necessarily surjective). 
Then, for any complex polynomial $\varphi$, we have
$$
\|\varphi(V)\colon Y \longrightarrow Y\| 
\leq \|\varphi({\mathfrak s}_{p,Y})\colon 
\ell^p_{\mathbb Z}(Y) \longrightarrow \ell^p_{\mathbb Z}(Y)\|.
$$
\end{lemma}

\begin{proof}
For simplicity, set $S = {\mathfrak s}_{p,Y}$. 
Let $P\colon \ell^p_{\mathbb Z}(Y) \to \ell^p_{\mathbb Z}(Y)$ 
be the norm-one projection that maps a sequence $(b_k)_{k\in\mathbb Z}$ to
$(\cdots, 0, 0, b_0, b_1, \ldots)$.

Fix a complex polynomial $\varphi$ and take $r\in(0,1)$. Define the operator $W_r\colon Y \to \ell^p_{\mathbb Z}(Y)$ by
$$
W_r(y) = (\cdots, 0, 0, y, rV(y), (rV)^2(y), \ldots), \qquad y \in Y,
$$
where $(rV)^k(y)$ sits in the $k$-th coordinate for $k\ge 0$. One checks directly that, for every $n\ge 0$,
$$
W_r (rV)^n = P S^n W_r.
$$
By linearity, this extends to
$$
W_r \varphi(rV) = P \varphi(S) W_r.
$$
Taking norms gives
$$
\|W_r \varphi(rV)(y)\| \le \|\varphi(S)\| \, \|W_r(y)\|, \qquad y \in Y.
$$

Since $V$ is an isometry, we have
$$
\|W_r(y)\| = \Bigl(\sum_{k=0}^\infty r^{pk} \|V^k(y)\|^p\Bigr)^{1/p}
= \Bigl(\sum_{k=0}^\infty r^{pk}\Bigr)^{1/p} \|y\|, \qquad y \in Y.
$$
Applying the same computation to $\varphi(rV)(y)$ yields
$$
\|\varphi(rV)(y)\| \le \|\varphi(S)\| \, \|y\|, \qquad y \in Y,
$$
and hence $\|\varphi(rV)\| \le \|\varphi(S)\|$. Letting $r \to 1$ completes the proof:
$$
\|\varphi(V)\| \le \|\varphi(S)\|.
\qedhere$$
\end{proof}

\begin{proof}[Proof of Theorem \ref{6SepMain}]
Let $T \in \C_p$. Then $T$ is a positive contraction. 
From the proof of Theorem \ref{5Main}, it is enough to verify the following two properties in order to obtain (1) and (2) almost surely:

\begin{itemize}
\item[(a)] For all $x \in L^p(M)$,
$$
\Bigl(\frac{1}{N}\sum_{n=1}^N T^n(x) - Q_T(x)\Bigr)_{N \ge 1} \in L^p(M; c_0);
$$

\item[(b)] There exists $\varepsilon > 0$ such that for almost every $\omega \in \Omega$, one can find a constant $C_\omega \ge 0$ with
$$
\frac{1}{W_N} \Bignorm{\sum_{n=1}^N Y_n(\omega)\, T^n}_{L^p \to L^p} \le C_\omega N^{-\varepsilon}, 
\qquad N \ge 1.
$$
\end{itemize}

Property (a) follows from \cite[Theorem 1.3]{HRW} together with the proof of \cite[Theorem 6.3]{JX}.

We turn to (b). By the definition of $\C_p$, there exists a sequence $(T_j)_{j \ge 1}$ converging to $T$ in the strong operator topology, where each $T_j$ is a convex combination of positive contractive separating maps. By \cite[Corollary 4.7]{HRW}, for every $j \ge 1$ and $N \ge 1$, there exists a tracial von Neumann algebra $(B_{j,N}, \tau_{j,N})$, contractive maps
$$
J_{j,N}\colon L^p(M) \to L^p(B_{j,N}), \qquad
P_{j,N}\colon L^p(B_{j,N}) \to L^p(M),
$$
and an isometry $V_{j,N}\colon L^p(B_{j,N}) \to L^p(B_{j,N})$ such that
$$
T_j^n = P_{j,N} V_{j,N}^n J_{j,N}, \qquad n = 0,1,\dots,N.
$$
Consequently, for any $\omega \in \Omega$ and all $j,N \ge 1$,
\begin{align*}
\frac{1}{W_N} \sum_{n=1}^N Y_n(\omega)\, T_j^n 
&= P_{j,N} \Biggl(\frac{1}{W_N} \sum_{n=1}^N Y_n(\omega)\, V_{j,N}^n\Biggr) J_{j,N},\\
\frac{1}{W_N} \Bignorm{\sum_{n=1}^N Y_n(\omega)\, T_j^n} 
&\le \frac{1}{W_N} \Bignorm{\sum_{n=1}^N Y_n(\omega)\, V_{j,N}^n}.
\end{align*}
Now consider the direct sums
$$
B := \bigoplus_{j,N \ge 1}^\infty B_{j,N}, \qquad
Y := \bigoplus_{j,N \ge 1}^p L^p(B_{j,N}),
$$
so that $Y = L^p(B)$, where $B$ carries the direct sum trace of the $\tau_{j,N}$. Applying Lemma \ref{6Transfer} to each isometry $V_{j,N}$ and then passing to the limit $j \to \infty$, we obtain
$$
\frac{1}{W_N} \Bignorm{\sum_{n=1}^N Y_n(\omega)\, T^n} \le 
\frac{1}{W_N} \Bignorm{\sum_{n=1}^N Y_n(\omega)\, {\mathfrak s}_{p,Y}^n}, 
\qquad \omega \in \Omega,\ N \ge 1.
$$
Finally, observe that the backward shift $S$ on $\ell^\infty_{\mathbb Z}(A)$ is a positive Dunford–Schwartz operator, and ${\mathfrak s}_{p,Y}$ is its $L^p$-extension $S_p$. Hence $S$ satisfies the assumptions of Theorem \ref{5Main}, and therefore satisfies (\ref{5LacLp}). Thus there exist $\varepsilon > 0$ and a full measure set $\Omega' \subset \Omega$ such that (\ref{5LacLp}) holds for all $\omega \in \Omega'$. This is exactly property (b).
\end{proof}

\begin{remark}\label{6Ack}
This remark concerns the commutative setting. Let $(\Sigma,\mu)$ be a measure space, let $1<p<\infty$, and let $T\colon L^p(\mu)\to L^p(\mu)$ be a positive contraction. By Akcoglu’s dilation theorem \cite{AS}, $T$ admits an isometric dilation on another $L^p$-space $L^p(\mu')$. This allows us to repeat, without change, the argument used in the proof of Theorem \ref{6SepMain}. In particular, there exists a measurable set $\Omega' \subset \Omega$ with $\Pdb(\Omega') = 1$ such that, for every $\omega \in \Omega'$ and every $x \in L^p(\Sigma)$, we have both
\[
\frac{1}{W_N}\sum_{n=1}^N X_n(\omega) T^n(x) \longrightarrow Q_T(x) 
\quad \text{and} \quad
\frac{1}{m}\sum_{k=1}^m T^{n_k(\omega)}(x) \longrightarrow Q_T(x),
\]
for $\mu$-almost every point in $\Sigma$.
\end{remark}

\bigskip\noindent
{\bf Acknowledgement.} The authors are grateful to Oleksiy Klurman for several stimulating and insightful discussions. This work was supported by a London Mathematical Society grant for the Support of Joint Research Groups (Scheme 3) [Ref: 32533], as well as by the Isaac Newton Institute and the Heilbronn Institute through the Additional Funding Programme for Mathematical Sciences, delivered by EPSRC (EP/V521917/1).1). 
This work was also supported by the EIPHI Graduate school (contract ANR-17-EURE-0002).

 \bigskip   
\end{document}